\documentclass{amsart}
\usepackage{amssymb, amsmath, amsthm, graphics, comment, xspace, enumerate}
\newtheorem{thm}{Theorem}[section]

\newtheorem{lem}[thm]{Lemma}
\newtheorem{prop}[thm]{Proposition}
\theoremstyle{definition}
\newtheorem{defn}[thm]{Definition}
\newtheorem{example}[thm]{Example}
\theoremstyle{remark}

\numberwithin{equation}{section}

\begin{document}
\title[Generalized Almost Automorphic Functions]{Almost Automorphic and Asymptotically Almost Automorphic Type Functions in Lebesgue Spaces with Variable Exponents
$L^{p(x)}$}

\author{Toka Diagana}
\address{Department of Mathematical Sciences, University of Alabama in Huntsville, 301  Sparkman Drive, Huntsville, 
AL 35899, USA}
\email{toka.diagana@uah.edu}

\author{Marko Kosti\' c}
\address{Faculty of Technical Sciences,
University of Novi Sad,
Trg D. Obradovi\' ca 6, 21125 Novi Sad, Serbia}
\email{marco.s@verat.net}

{\renewcommand{\thefootnote}{} \footnote{
\\ \text{  }  \ \    {\it Key words and phrases.} Lebesgue spaces with variable exponents,
Stepanov almost automorphy with variable exponents,  asymptotical Stepanov almost automorphy with variable exponents, abstract Volterra integro-differential inclusions, abstract (multi-term) fractional differential equations.
\\  \text{  }  \ \ The second named author is partially supported by grant 174024 of Ministry
of Science and Technological Development, Republic of Serbia.}}

\begin{abstract}
The paper introduces and studies the class of (asymptotically) Stepanov almost automorphic functions with variable exponents. 
Any function belonging this class needs to be (asymptotically) Stepanov almost automorphic.
A few relevant applications to abstract Volterra integro-differential inclusions in Banach spaces is presented.
\end{abstract}
\maketitle

\section{Introduction and Preliminaries}\label{intro1}

The main aim of this paper is to continue our recent research of Stepanov $p(x)$-almost periodicity and asymptotical Stepanov $p(x)$-almost periodicity raised in \cite{toka-marek}, as well as to initiate the study of generalized almost automorphy and generalized asymptotical almost automorphy that intermediate the classical and Stepanov concept. This is done here by examining the notion of Stepanov $p(x)$-almost automorphy and asymptotical Stepanov $p(x)$-almost automorphy.  We basically follow the approach obeyed in \cite{toka-marek}, which enables us to conclude that the introduced classes of functions are translation invariant (Stepanov-like pseudo-almost automorphic functions with variable exponents, which have been analyzed in \cite{m-zitane-prim}, do not possess this property).

We investigate generalized almost automorphic and generalized asymptotically almost automorphic type functions in Banach spaces by means of results from the theory of Lebesgue spaces with variable exponents
$L^{p(x)}.$ For a given measurable function $p : [0,1] \rightarrow [1,\infty],$ we define the notions of an $S^{p(x)}$-almost automorphic function and an asymptotically $S^{p(x)}$-almost automorphic function. In the case that $p(x)\equiv p \geq 1,$ the introduced notion is equivalent to the usually considered notion of $S^{p}$-almost automorphy and asymptotical $S^{p}$-almost automorphy. 

The organization and main ideas of this paper are briefly described as follows. In Subsection 1.1, Subsection 1.2 and Subsection 1.3, we collect the basic facts about fractional calculus, multivalued linear operators
and Lebesgue spaces with variable exponents
$L^{p(x)},$ respectively. Section \ref{section2}
is devoted to the recapitulation of some basic definitions and results about generalized almost periodic and generalized almost automorphic functions. We start Section \ref{section3} by recalling the definitions of Stepanov $p(x)$-boundedness and Stepanov $p(x)$-almost periodicity in the sense of \cite{toka-marek}. The notion of (asymptotical) Stepanov $p(x)$-almost automorphy is introduced in Definition \ref{sasasa} (Definition \ref{gaston-toka}). It is expected that the notion of (asymptotical) Stepanov $p(x)$-almost automorphy is much more general than that of (asymptotical) Stepanov $p(x)$-almost periodicity, and we explictly show this in Proposition \ref{oze} and Proposition \ref{ozeze}. Several continuous embeddings between various Stepanov $p(x)$-almost automorphic spaces are proved in Theorem \ref{toka-maremare}, where it is particularly shown that an $S^{p(x)}$-almost automorphic function has to be Stepanov $1$-almost automorphic. 

We know that 
any almost periodic function has to be $S^{p(x)}$-almost periodic for any measurable function $p : [0,1] \rightarrow [1,\infty].$ This is no longer true for almost automorphy, where we perceive some peculiar differences between almost automorphy and compact almost automorphy, proving that the almost automorphy of a function 
$f : {\mathbb R} \rightarrow X$ implies its $S^{p(x)}$-almost automorphy only if we impose the validity of some additional conditions (see Proposition \ref{propa}); all these statements have natural reformulations for asymptotical $S^{p(x)}$-almost automorphy.

In Section 4, we introduce (asymptotically) Stepanov $p(x)$-almost automorphic functions depending on two parameters and formulate a great number of related composition principles, providing thus slight extensions of results obtained in \cite{ding-xll}, \cite{fan-et-ali} and \cite{nova-mono}. Keeping this in mind, it is very technical to reword several known results concerning semilinear analogues of the inclusions \eqref{decko-leftt}-\eqref{decko-left} and (DFP)$_{f,\gamma}$ considered below (see e.g. \cite[Theorem 4-Theorem 8; Theorem 10]{element} for more details in this direction). Because of that, in this paper, we will not consider semilinear Cauchy inclusions.

Concerning applications, our main results are given in Section \ref{raske}, where we analyze the invariance of generalized (asymptotical) almost automorphy in Lebesgue spaces with variable exponents
$L^{p(x)}$ under the actions of convolution products (see Proposition \ref{ravi-and-variable} and Proposition \ref{stepanov-almost-automorphy-p(x)}). Although strengthens some previous results of ours in this direction, we feel duty bound to say that it is very difficult to apply Proposition \ref{ravi-and-variable} in the case that $p(x)$ is not a  constant function. This is no longer case with the assertion of 
Proposition \ref{stepanov-almost-automorphy-p(x)}, where the use of 
ergodic Stepanov components with variable exponents plays a crucial role (see also Example \ref{gotyie} below).
In addition to the above, we propose several open problems, questions, illustrative examples and applications of our abstract results. 

We use the standard notation throughout the paper.
We assume
that $(X,\| \cdot \|)$ is a complex Banach space. If $Y$ is also such a space, then we denote by
$L(X,Y)$ the space of all continuous linear mappings from $X$ into
$Y;$ $L(X)\equiv L(X,X).$ Assuming $A$ is a closed linear operator
acting on $X,$
then the domain, kernel space and range of $A$ will be denoted by
$D(A),$ $N(A)$ and $R(A),$
respectively. 

Let $I={\mathbb R}$ or $I=[0,\infty).$ By $C_{b}(I: X)$ we denote the Banach space consisting of all bounded continuous functions
$I\mapsto X,$ equipped with the sup-norm.
The Gamma function is denoted by
$\Gamma(\cdot)$ and the principal branch is always used to take
the powers; the convolution like
mapping $\ast$ is given by $f\ast g(t):=\int_{0}^{t}f(t-s)g(s)\,
ds .$ Set $g_{\zeta}(t):=t^{\zeta-1}/\Gamma(\zeta),$ $\zeta>0.$ For any $s\in {\mathbb R},$ we define $\lfloor s \rfloor :=\sup \{
l\in {\mathbb Z} : s\geq l \}$ and $\lceil s \rceil :=\inf \{ l\in
{\mathbb Z} : s\leq l \}.$

\subsection{Fractional Calculus}\label{na-smert}

The first conference on fractional calculus and fractional differential equations was held in New Haven (1974). Since then, fractional calculus has gained more
and more attention due to its wide applications in various fields of science, such as mathematical physics,
engineering, biology, aerodynamics, chemistry, economics etc. Fairly complete information about fractional calculus and fractional
differential equations can be obtained
by consulting
\cite{bajlekova}, \cite{Diet}, \cite{kilbas}, \cite{knjigaho} and references cited therein.

In this subsection, we will briefly explain the types of fractional derivatives which will be used in the paper. Essentially, we use only the Caputo fractional derivatives and Weyl-Liouville fractional derivatives of order
$\gamma \in (0,1].$ They are defined as follows. 

Let $\gamma \in (0,1).$ 
Then
the Caputo fractional derivative\index{fractional derivatives!Caputo}
${\mathbf D}_{t}^{\gamma}u(t)$ is defined for those functions
$u :[0,\infty) \rightarrow X$ satisfying that, for every $T>0,$ we have $u _{| (0,T]}(\cdot) \in C((0,T]: X),$ $u(\cdot)-u(0) \in L^{1}((0,T) : X)$
and $g_{1-\gamma}\ast (u(\cdot)-u(0)) \in W^{1,1}((0,T) : X),$
by
$$
{\mathbf
D}_{t}^{\gamma}u(t)=\frac{d}{dt}\Biggl[g_{1-\gamma}
\ast \Bigl(u(\cdot)-u(0)\Bigr)\Biggr](t),\quad t\in (0,T];
$$
see \cite[p. 7]{bajlekova} for the notion of Sobolev space $W^{1,1}((0,T) : X).$
The Weyl-Liouville fractional derivative
$D_{t,+}^{\gamma}u(t)$ of order $\gamma$ is defined for those continuous functions
$u : {\mathbb R} \rightarrow X$
such that $t\mapsto \int_{-\infty}^{t}g_{1-\gamma}(t-s)u(s)\, ds,$ $t\in {\mathbb R}$ is a well-defined continuously differentiable mapping, by
$$
D_{t,+}^{\gamma}u(t):=\frac{d}{dt}\int_{-\infty}^{t}g_{1-\gamma}(t-s)u(s)\, ds,\quad t\in {\mathbb R}.
$$
Set $
{\mathbf D}_{t}^{1}u(t):=(d/dt)u(t)$ and $
D_{t,+}^{1}u(t):=-(d/dt)u(t).$

\subsection{Multivalued linear operators}\label{na-smrt}

We need some basic definitions and results about multivalued linear operators in Banach spaces (see \cite{toka-marek} for more details in this direction). Suppose that $X$ and $Y$ are two Banach spaces.
A multivalued map (multimap) ${\mathcal A} : X \rightarrow P(Y)$ is said to be a multivalued
linear operator, MLO for short, iff the following holds:
\begin{itemize}
\item[(i)] $D({\mathcal A}) := \{x \in X : {\mathcal A}x \neq \emptyset\}$ is a linear subspace of $X$;
\item[(ii)] ${\mathcal A}x +{\mathcal A}y \subseteq {\mathcal A}(x + y),$ $x,\ y \in D({\mathcal A})$
and $\lambda {\mathcal A}x \subseteq {\mathcal A}(\lambda x),$ $\lambda \in {\mathbb C},$ $x \in D({\mathcal A}).$
\end{itemize}
In the case that $X=Y,$ then we say that ${\mathcal A}$ is an MLO in $X.$
It is well known that
for any $x,\ y\in D({\mathcal A})$ and $\lambda,\ \eta \in {\mathbb C}$ with $|\lambda| + |\eta| \neq 0,$ we
have $\lambda {\mathcal A}x + \eta {\mathcal A}y = {\mathcal A}(\lambda x + \eta y).$ If ${\mathcal A}$ is an MLO, then ${\mathcal A}0$ is a linear manifold in $Y$
and ${\mathcal A}x = f + {\mathcal A}0$ for any $x \in D({\mathcal A})$ and $f \in {\mathcal A}x.$ Define 
the range $R({\mathcal A})$ of ${\mathcal A}$ 
by 
$R({\mathcal A}):=\{{\mathcal A}x :  x\in D({\mathcal A})\}.$ 

Let ${\mathcal A}$ be an MLO in $X$. 
Then the resolvent set of ${\mathcal A},$ $\rho({\mathcal A})$ for short, is defined as the union of those complex numbers
$\lambda \in {\mathbb C}$ for which
\begin{itemize}
\item[(i)] $X= R(\lambda-{\mathcal A})$;
\item[(ii)] $(\lambda - {\mathcal A})^{-1}$ is a single-valued linear continuous operator on $X.$
\end{itemize}
The operator $\lambda \mapsto (\lambda -{\mathcal A})^{-1}$ is called the resolvent of ${\mathcal A}$ ($\lambda \in \rho ({\mathcal A})$). Set $R(\lambda : {\mathcal A})\equiv  (\lambda -{\mathcal A})^{-1}$  ($\lambda \in \rho({\mathcal A})$).

Henceforward, we will employ the following condition:

\begin{itemize}
\item[(P)]
There exist finite constants $c,\ M>0$ and $\beta \in (0,1]$ such that
$$
\Psi:=\Bigl\{ \lambda \in {\mathbb C} : \Re \lambda \geq -c\bigl( |\Im \lambda| +1 \bigr) \Bigr\} \subseteq \rho({\mathcal A})
$$
and
$$
\| R(\lambda : {\mathcal A})\| \leq M\bigl( 1+|\lambda|\bigr)^{-\beta},\quad \lambda \in \Psi.
$$
\end{itemize}

\subsection{Lebesgue spaces with variable exponents
$L^{p(x)}$}\label{karambita}

Assume $\emptyset \neq \Omega \subseteq {\mathbb R}.$
By $M(\Omega  : X)$ we denote the collection of all measurable functions $f: \Omega \rightarrow X;$ the symbol $M(\Omega)$ stands for the collection of all functions $f\in M(\Omega : {\mathbb C})$ such that $f(x)\in {\mathbb R}$ for all $x\in \Omega.$ Furthermore, ${\mathcal P}(\Omega)$ denotes the vector space of all Lebesgue measurable functions $p : \Omega \rightarrow [1,\infty].$
For any $p\in {\mathcal P}(\Omega)$ and $f\in M(\Omega : X),$ set
$$
\varphi_{p(x)}(t):=\left\{
\begin{array}{l}
t^{p(x)},\ t\geq 0,\ 1\leq p(x)<\infty,\\
0,\ 0\leq t\leq 1,\ p(x)=\infty,\\
\infty,\ t>1,\ p(x)=\infty 
\end{array}
\right.
$$
and
$$
\rho(f):=\int_{\Omega}\varphi_{p(x)}(\|f(x)\|)\, dx .
$$
We define Lebesgue space 
$L^{p(x)}(\Omega : X)$ with variable exponent
as follows
$$
L^{p(x)}(\Omega : X):=\Bigl\{f\in M(\Omega : X): \lim_{\lambda \rightarrow 0+}\rho(\lambda f)=0\Bigr\}.
$$
Then 
$$
L^{p(x)}(\Omega : X)=\Bigl\{f\in M(\Omega : X):  \mbox{ there exists }\lambda>0\mbox{ such that }\rho(\lambda f)<\infty\Bigr\};
$$
see \cite[p. 73]{variable}.
For every $u\in L^{p(x)}(\Omega : X),$ we introduce the Luxemburg norm of $u(\cdot)$ in the following manner
$$
\|u\|_{p(x)}:=\|u\|_{L^{p(x)}(\Omega :X)}:=\inf\Bigl\{ \lambda>0 : \rho(f/\lambda)    \leq 1\Bigr\}.
$$
Equipped with the above norm, the space $
L^{p(x)}(\Omega : X)$ becomes a Banach one (see e.g. \cite[Theorem 3.2.7]{variable} for scalar-valued case), coinciding with the usual Lebesgue space $L^{p}(\Omega : X)$ in the case that $p(x)=p\geq 1$ is a constant function.
For any $p\in M(\Omega),$ we set 
$$
p^{-}:=\text{essinf}_{x\in \Omega}p(x) \ \ \mbox{ and } \ \ p^{+}:=\text{esssup}_{x\in \Omega}p(x).
$$
Define
$$
C_{+}(\Omega ):=\bigl\{ p\in M(\Omega): 1<p^{-}\leq p(x) \leq p^{+} <\infty \mbox{ for a.e. }x\in \Omega \bigr \}
$$
and
$$
D_{+}(\Omega ):=\bigl\{ p\in M(\Omega): 1 \leq p^{-}\leq p(x) \leq p^{+} <\infty \mbox{ for a.e. }x\in \Omega \bigr \}.
$$
Set 
$$
E^{p(x)}(\Omega :X):=\Bigl\{ f\in L^{p(x)}(\Omega :X) : \mbox{ for all }\lambda>0\mbox{ we have }\rho(\lambda f)<\infty\Bigr\};
$$
$E^{p(x)}(\Omega )\equiv E^{p(x)}(\Omega : {\mathbb C}).$ It is well known that 
$E^{p(x)}(\Omega :X)=
L^{p(x)}(\Omega :X),$ provided that $p\in D_{+}(\Omega )$ (see e.g. \cite{fan-zhao}).

We will use the following lemma (see e.g. \cite[Lemma 3.2.20, (3.2.22); Corollary 3.3.4; Lemma 3.2.8(c)]{variable} for scalar-valued case):

\begin{lem}\label{aux}
\begin{itemize}
\item[(i)] Let $p,\ q,\ r \in {\mathcal P}(\Omega),$ and let
$$
\frac{1}{q(x)}=\frac{1}{p(x)}+\frac{1}{r(x)},\quad x\in \Omega .
$$
Then, for every $u\in L^{p(x)}(\Omega : X)$ and $v\in L^{r(x)}(\Omega),$ we have $uv\in L^{q(x)}(\Omega : X)$
and
\begin{align*}
\|uv\|_{q(x)}\leq 2 \|u\|_{p(x)}\|v\|_{r(x)}.
\end{align*}
\item[(ii)] Let $\Omega $ be of a finite Lebesgue's measure, let $p,\ q \in {\mathcal P}(\Omega),$ and let $q\leq p$ a.e. on $\Omega.$ Then
 $L^{p(x)}(\Omega : X)$ is continuously embedded in $L^{q(x)}(\Omega : X).$
\item[(iii)] Let $p\in {\mathcal P}(\Omega),$ and let $f_{k},\ f\in M(\Omega  : X)$ for all $k\in {\mathbb N}.$ If $\lim_{k\rightarrow \infty} f_{k}(x)=f(x)$ for a.e. $x\in \Omega $
and there exists a real valued function $g\in E^{p(x)}(\Omega)$ such that $\|f_{k}(x)\|\leq g(x)$ for a.e. $x\in \Omega,$ then $\lim_{k\rightarrow \infty}\|f_{k}-f\|_{L^{p(x)}(\Omega :X)}=0.$
\end{itemize}
\end{lem}

For more details about Lebesgue spaces with variable exponents
$L^{p(x)},$ the reader may consult \cite{m-zitane}-\cite{m-zitane-prim}, \cite{variable}-\cite{fan-zhao} and \cite{doktor}.

\section{Generalized almost periodic and generalized almost automorphic functions}\label{section2}

Let $1\leq p <\infty,$ and let $f,\ g\in L^{p}_{loc}(I :X),$ where $I={\mathbb R}$ or $I=[0,\infty).$ We define the Stepanov `metric' by\index{Stepanov metric}
\begin{align*}
D_{S}^{p}\bigl[f(\cdot),g(\cdot)\bigr]:= \sup_{x\in I}\Biggl[\int_{x}^{x+1}\bigl \| f(t) -g(t)\bigr\|^{p}\, dt\Biggr]^{1/p}.
\end{align*}
The Stepanov norm of $f(\cdot)$ is introduced by\index{Stepanov norm} \index{Stepanov distance} \index{Weyl!distance} \index{Weyl!norm}
$
\| f \|_{S^{p}}:= D_{S}^{p}[f(\cdot),0].
$
It is said that a function $f\in L^{p}_{loc}(I :X)$ is Stepanov $p$-bounded, $S^{p}$-bounded shortly, iff\index{function!Stepanov bounded}
$$
\|f\|_{S^{p}}:=\sup_{t\in I}\Biggl( \int^{t+1}_{t}\|f(s)\|^{p}\, ds\Biggr)^{1/p}=\sup_{t\in I}\Biggl( \int^{1}_{0}\|f(s+t)\|^{p}\, ds\Biggr)^{1/p}<\infty.
$$
Furnished with the above norm, the space $L_{S}^{p}(I:X)$ consisted of all $S^{p}$-bounded functions is a Banach one. 
We refer the reader to \cite{toka-marek} for the notions of almost periodic functions and Stepanov $p$-almost periodic functions (see also \cite{diagana} and \cite{nova-mono}).  

Let $f : {\mathbb R} \rightarrow X$ be continuous. As it is well known,\index{function!almost automorphic}
$f(\cdot)$ is called almost automorphic, a.a. for short, iff for every real sequence $(b_{n})$ there exist a subsequence $(a_{n})$ of $(b_{n})$ and a map $g : {\mathbb R} \rightarrow X$ such that
\begin{align}\label{first-equ}
\lim_{n\rightarrow \infty}f\bigl( t+a_{n}\bigr)=g(t)\ \mbox{ and } \  \lim_{n\rightarrow \infty}g\bigl( t-a_{n}\bigr)=f(t),
\end{align}
pointwise for $t\in {\mathbb R}.$ If this is the case, $f\in C_{b}({\mathbb R} : X)$ and the limit function $g(\cdot)$ must be bounded on ${\mathbb R}$ but not necessarily continuous on ${\mathbb R}.$ It is said that $f(\cdot)$ is compactly almost automorphic iff
the convergence in \eqref{first-equ} is uniform on compacts of ${\mathbb R}.$
The vector space consisting of all almost automorphic, resp., compactly almost automorphic functions, is denoted by $AA({\mathbb R} :X),$ resp., $AA_{c}({\mathbb R} :X).$ By Bochner's criterion \cite{diagana}, any almost periodic function has to be compactly almost automorphic. 

The space of pseudo-almost automorphic functions, denoted by $PAA({\mathbb R} : X),$ is defined as the direct sum of spaces $AA({\mathbb R} : X)$
and $PAP_{0}({\mathbb R} : X),$ where
$PAP_{0}({\mathbb R} : X)$ denotes the space consisting of all bounded continuous functions $\Phi : {\mathbb R} \rightarrow X$ such
that
$$
\lim_{r\rightarrow \infty} \frac{1}{2r}\int^{r}_{-r}\| \Phi(s)\|\, ds=0.
$$
Equipped with the sup-norm, the space $PAA({\mathbb R} : X)$ is a Banach one.

Following G. M. N'Gu\' er\' ekata and A. Pankov \cite{gaston-apankov}, we say that a function $f\in L_{loc}^{p}({\mathbb R}:X)$ is said to be Stepanov $p$-almost automorphic, $S^{p}$-almost automorphic or $S^{p}$-a.a. shortly, iff for
every real sequence $(a_{n}),$ there exists a subsequence $(a_{n_{k}})$ \index{ Stepanov almost automorphy}
and a function $g\in L_{loc}^{p}({\mathbb R}:X)$ such that
\begin{align}\label{ne-mu-je-stepanov}
\lim_{k\rightarrow \infty}\int^{t+1}_{t}\Bigl \| f\bigl(a_{n_{k}}+s\bigr) -g(s)\Bigr \|^{p} \, ds =0
\end{align}
and
\begin{align}\label{ne-mu-je-stepanovi}
\lim_{k\rightarrow \infty}\int^{t+1}_{t}\Bigl \| g\bigl( s-a_{n_{k}}\bigr) -f(s)\Bigr \|^{p} \, ds =0
\end{align}
for each $ t\in {\mathbb R}$; a function $f\in L_{loc}^{p}([0,\infty):X)$ is said to be asymptotically Stepanov $p$-almost automorphic, asymptotically $S^{p}$-a.a. shortly, iff there exists an $S^{p}$-almost automorphic
function $g(\cdot)$ and a function $q\in  L_{S}^{p}([0,\infty): X)$ such that $f(t)=g(t)+q(t),$ $t\geq 0$ and $\hat{q}\in C_{0}([0,\infty) : L^{p}([0,1]: X));$ any $S^{p}$-almost automorphic function $f(\cdot)$ has to be $S^{p}$-bounded ($1\leq p<\infty$); here and hereafter, $\hat{q}(t):=q(t+\cdot),$ $t\geq 0.$ The vector space consisting of all $S^{p}$-almost automorphic functions, resp., asymptotically $S^{p}$-almost automorphic functions, will be denoted by $ AAS^{p}({\mathbb R} : X),$ resp., $ AAAS^{p}([0,\infty) : X).$

If $1\leq p<q<\infty$ and $f(\cdot)$ is (asymptotically) Stepanov $q$-almost automorphic, then $f(\cdot)$ is (asymptotically) Stepanov $p$-almost automorphic. Therefore, the (asymptotical) Stepanov $p$-almost automorphy of $f(\cdot)$ for some $p\in [1,\infty)$ implies the (asymptotical) Stepanov $1$-almost automorphy of $f(\cdot).$
It is a well-known fact that if $f(\cdot)$ is an almost automorphic (a.a.a.) function
then $f(\cdot)$ is also $S^p$-almost automorphic (asymptotically $S^p$-a.a.) for $1\leq p <\infty.$ The converse statement is false, however.

A function $f(\cdot)$ is said to be (asymptotically) Stepanov almost periodic (automorphic) iff $f(\cdot)$ is (asymptotically) Stepanov $1$-almost periodic (automorphic).

\section{Generalized almost automorphic type functions in  Lebesgue spaces with variable exponents
$L^{p(x)}$}\label{section3}

The following notion of Stepanov $p(x)$-boundedness has been recently introduced in \cite{toka-marek} by using a completely different approach from that one employed in former papers by T. Diagana and M. Zitane (cf. \cite[Definition 3.10]{m-zitane} and \cite[Definition 4.5]{m-zitane-prim}):

\begin{defn}\label{daniel-toka}
Let $p\in {\mathcal P}([0,1]),$ and let $I={\mathbb R}$ or $I=[0,\infty).$ Then it is said that a function $f\in M(I : X)$ is Stepanov $p(x)$-bounded, $S^{p(x)}$-bounded in short, iff $f(\cdot +t) \in L^{p(x)}([0,1]: X)$ for all $t\in I,$ and $\sup_{t\in I} \|f(\cdot +t)\|_{p(x)}<\infty ,$ i.e.,
$$
\|f\|_{S^{p(x)}}:=\sup_{t\in I}\inf\Biggl\{ \lambda>0 : \int_{0}^{1}\varphi_{p(x)}\Biggl( \frac{\|f(x +t)\|}{\lambda}\Biggr)\, dx \leq 1\Biggr\}<\infty.
$$
By $L_{S}^{p(x)}(I:X)$ we denote the vector space consisting of all such functions.
\end{defn}

Denote by $\hookrightarrow$
a continuous embedding between normed spaces. Furnished with the norm $\|\cdot\|_{S^{p(x)}}$, the space $L_{S}^{p(x)}(I:X)$ consisted of all $S^{p(x)}$-bounded functions is a Banach one and we have $L_{S}^{p(x)}(I:X) \hookrightarrow L_{S}^{1}(I:X),$ for any $p\in {\mathcal P}([0,1]).$ The space $L_{S}^{p(x)}(I:X)$ is translation invariant in the sense that, for every $f\in L_{S}^{p(x)}(I:X)$ and $\tau \in I,$ we have $f(\cdot+\tau) \in L_{S}^{p(x)}(I:X).$ 

In \cite{toka-marek}, we have introduced the concept of (asymptotical) $S^{p(x)}$-almost periodicity as follows:

\begin{defn}\label{sasasa-ap}
\begin{itemize}
\item[(i)]
Let $p\in {\mathcal P}([0,1]),$ and let $I={\mathbb R}$ or $I=[0,\infty).$ Then it is said that a function $f\in L_{S}^{p(x)}(I:X)$ is Stepanov $p(x)$-almost periodic, Stepanov $p(x)$-a.p. in short, iff the function $\hat{f} : I \rightarrow L^{p(x)}([0,1]: X)$ is almost periodic. By $APS^{p(x)}(I : X)$ we denote the vector space consisting of all such functions.
\item[(ii)] Let $p\in {\mathcal P}([0,1]),$ and let $I=[0,\infty).$ Then it is said that a function $f\in L_{S}^{p(x)}(I:X)$ is asymptotically Stepanov $p(x)$-almost periodic,  asymptotically Stepanov $p(x)$-a.p. in short, iff the function $\hat{f} : I \rightarrow L^{p(x)}([0,1]: X)$ is  asymptotically almost periodic. By $AAPS^{p(x)}(I : X)$ we denote the vector space consisting of all such functions; the abbreviation $S^{p(x)}_{0}([0,\infty):X)$
will be used to denote the set of all functions
$q\in  L_{S}^{p(x)}([0,\infty): X)$ such that $\hat{q}\in C_{0}([0,\infty) : L^{p(x)}([0,1]:X)).$
\end{itemize}
\end{defn}

We know that the space $APS^{p(x)}(I : X)$ is translation invariant in the sense that, for every $f\in APS^{p(x)}(I : X)$ and $\tau \in I,$ we have $f(\cdot+\tau) \in APS^{p(x)}(I : X).$ A similar statement holds for the space $AAPS^{p(x)}([0,\infty) : X).$

Now we introduce the concept of $S^{p(x)}$-almost automorphy as follows:

\begin{defn}\label{sasasa}
Let $p\in {\mathcal P}([0,1]).$ Then it is said that a function $f\in L_{S}^{p(x)}({\mathbb R}:X)$ is Stepanov $p(x)$-almost automorphic, Stepanov $p(x)$-a.a. in short, iff for
every real sequence $(a_{n}),$ there exists a subsequence $(a_{n_{k}})$ 
and a function $g\in L_{S}^{p(x)}({\mathbb R}:X)$ such that
\begin{align}\label{ne-mu-je-stepanov-p(x)}
\lim_{k\rightarrow \infty}\bigl \| f\bigl(a_{n_{k}}+\cdot+t\bigr) -g(\cdot+t)\bigr \|_{L^{p(x)}([0,1]: X)}  =0
\end{align}
and
\begin{align*}
\lim_{k\rightarrow \infty}\bigl \| g\bigl( \cdot-a_{n_{k}}+t\bigr) -f(\cdot+t)\bigr \|_{L^{p(x)}([0,1]: X)}  =0
\end{align*}
for each $ t\in {\mathbb R}.$
By $AAS^{p(x)}({\mathbb R} : X)$ we denote the vector subspace of $L_{S}^{p(x)}({\mathbb R}:X)$ consisting of all such functions.
\end{defn}

For asymptotical $S^{p(x)}$-almost automorphy, we will use the following notion:

\begin{defn}\label{gaston-toka}
Let $p\in {\mathcal P}([0,1]).$ 
A function $f\in L_{S}^{p(x)}([0,\infty):X)$ is said to be asymptotically Stepanov $p(x)$-almost automorphic, asymptotically $S^{p(x)}$-a.a. shortly, iff there exist an $S^{p(x)}$-almost automorphic
function $g \in AAS^{p(x)}({\mathbb R} : X)$ and a function $q\in  L_{S}^{p(x)}([0,\infty): X)$ such that $f(t)=g(t)+q(t),$ $t\geq 0$ and $\hat{q}\in C_{0}([0,\infty) : L^{p(x)}([0,1]: X)).$ 
\end{defn}

It follows immediately from definition that the spaces $AAS^{p(x)}({\mathbb R} : X)$ and $AAAS^{p(x)}([0,\infty) : X)$ are translation invariant, with the meaning clear. Furthermore,
it can be simply checked that the notions of Stepanov $p(x)$-boundedness and (asymptotical) Stepanov $p(x)$-almost automorphy are equivalent with those ones introduced in the previous section, provided that $p(x)\equiv p\geq 1$ is a constant function. Furthermore, the following holds:

\begin{prop}\label{oze}
Let $p\in {\mathcal P}([0,1])$ and let $f : {\mathbb R} \rightarrow X$ be $S^{p(x)}$-almost periodic. Then $f(\cdot)$ is $S^{p(x)}$-almost automorphic. 
\end{prop}

\begin{proof}
Let $(a_{n})$ be a given real sequence. By Bochner's criterion \cite{diagana}, there exists a subsequence $(a_{n_{k}})$ of $(a_{n})$ and a uniformly continuous bounded function $G : {\mathbb R} \rightarrow L^{p(x)}([0,1] : X)$ such that
\begin{align}\label{crfl}
\lim_{k\rightarrow \infty}\sup_{t\in {\mathbb R}}\Bigl\| f\bigl(t+a_{n_{k}}+\cdot)-G(t)(\cdot)\Bigr \|_{L^{p(x)}([0,1] : X)}=0.
\end{align}
It suffices to show that there exists a function $g : {\mathbb R} \rightarrow L^{p(x)}([0,1] : X)$ such that $G(t)(s)=g(t+s)$ for any $t\in {\mathbb R}$ and a.e. $s\in [0,1].$ We define $g(\cdot):=G(\lfloor \cdot \rfloor)(\cdot-\lfloor \cdot \rfloor).$ Then it is clear that, for every $t\in {\mathbb Z},$ we have $G(t)(\cdot)=g(t+\cdot)$ a.e. on $[0,1].$ Suppose that $t\notin {\mathbb Z}.$ Since, clearly, $g : {\mathbb R} \rightarrow L^{p(x)}([0,1] : X),$ we only need to prove that $G(t)(s)=G(\lfloor t \rfloor)(t-\lfloor t \rfloor+s)$ for a.e. $s\in (0,\lceil t \rceil -t)$ and $G(t)(s)=G(\lceil t \rceil)(t-\lceil t \rceil +s)$ for a.e. $s\in (\lceil t \rceil -t,1).$ For the sake of brevity, we will prove the validity of second equality. Since $\int_{\lceil t \rceil -t}^{1}\cdot \cdot \cdot \leq \int^{1}_{0}\cdot \cdot \cdot,$ \eqref{crfl} implies that 
\begin{align*}
& \inf\Biggl\{ \lambda >0  : \int^{1}_{\lceil t \rceil -t}\varphi_{p(x)}\Biggl( \frac{\| f(t+a_{n_{k}}+x)-G(t)(x)\|}{\lambda} \Biggr)\, dx\leq 1\Biggr\} 
\\ & =\inf\Biggl\{ \lambda >0  : \int^{1+t-\lceil t \rceil}_{0}\varphi_{p(x)}\Biggl( \frac{\| f(\lceil t \rceil +a_{n_{k}}+x)-G(t)(x-\lceil t \rceil+t)\|}{\lambda} \Biggr)\, dx\leq 1\Biggr\} \rightarrow 0,
\end{align*}  
as $k\rightarrow \infty.$ On the other hand, by \eqref{crfl} and $\int^{1+t-\lceil \rceil}_{0}\cdot \cdot \cdot \leq \int^{1}_{0}\cdot \cdot \cdot,$ we have
\begin{align*}
\inf\Biggl\{ \lambda >0  : \int^{1+t-\lceil t \rceil}_{0}\varphi_{p(x)}\Biggl( \frac{\| f(\lceil t \rceil +a_{n_{k}}+x)-G(t)(x-\lceil t \rceil+t)\|}{\lambda} \Biggr)\, dx\leq 1\Biggr\} \rightarrow 0,
\end{align*}
as $k\rightarrow \infty.$ The uniqueness of limits in the space $L^{p(x)}([0,1+t-\lceil t \rceil] : X)$ yields the required equality.
\end{proof}

Making use of Proposition \ref{oze} and \cite[Proposition 3.12]{toka-marek}, we immediately get: 

\begin{prop}\label{ozeze}
Let $p\in {\mathcal P}([0,1])$ and let $f : [0,\infty) \rightarrow X$ be asymptotically $S^{p(x)}$-almost periodic. Then $f(\cdot)$ is asymptotically $S^{p(x)}$-almost automorphic. 
\end{prop}

Assume that $p\in {\mathcal P}([0,1]),$ resp. $p\in D_{+}([0,1]),$ and $1 \leq p^{-}\leq p(x) \leq p^{+} <\infty$ for a.e. $x\in [0,1] .$ Then Lemma \ref{aux}(ii) implies that $L^{\infty}({\mathbb R}: X) \hookrightarrow L^{p(x)}_{S}({\mathbb R} : X)\hookrightarrow L^{1}_{S}({\mathbb R}: X),$ resp., $L^{p^{+}}_{S}({\mathbb R} : X)\hookrightarrow  L^{p(x)}_{S}({\mathbb R} : X) \hookrightarrow L^{p^{-}}_{S}({\mathbb R} : X).$ 
Therefore, a similar line of reasoning as in almost periodic case shows that the following theorem holds true; for the sake of completeness, we will prove only the second part of (i), $AAAS^{p(x)}( [0,\infty) : X) \hookrightarrow AAAS^{1}( [0,\infty) : X)$:
  
\begin{thm}\label{toka-maremare}
\begin{itemize}
\item[(i)] Let $p\in {\mathcal P}([0,1]).$ Then $AAS^{p(x)}({\mathbb R} : X) \hookrightarrow AAS^{1}({\mathbb R} : X)$ and $AAAS^{p(x)}( [0,\infty) : X) \hookrightarrow AAAS^{1}( [0,\infty) : X).$  
\item[(ii)] Let $p\in D_{+}([0,1])$ and $1 \leq p^{-}\leq p(x) \leq p^{+} <\infty$ for a.e. $x\in [0,1] $. Then 
$AAS^{p^{+}}({\mathbb R}: X)\hookrightarrow  AAS^{p(x)}({\mathbb R}: X) \hookrightarrow AAS^{p^{-}}({\mathbb R} : X)$ and $AAAS^{p^{+}}([0,\infty): X)\hookrightarrow  AAAS^{p(x)}([0,\infty): X) \hookrightarrow AAAS^{p^{-}}([0,\infty) : X).$
\item[(iii)] If $p,\ q\in {\mathcal P}([0,1])$ and $p\leq q$ a.e. on $[0,1],$ then
$
AAS^{q(x)}({\mathbb R} : X) \hookrightarrow AAS^{p(x)}({\mathbb R} : X)
$
and $
AAAS^{q(x)}([0,\infty) : X) \hookrightarrow AAAS^{p(x)}([0,\infty) : X).
$
\end{itemize}
\end{thm}

\begin{proof}
Let $f\in AAAS^{p(x)}( [0,\infty) : X) .$ By definition, there exist an $S^{p(x)}$-almost automorphic
function $g(\cdot)$ and a function $q\in  L_{S}^{p(x)}([0,\infty): X)$ such that $f(t)=g(t)+q(t),$ $t\geq 0$ and $\hat{q}\in C_{0}([0,\infty) : L^{p(x)}([0,1]: X)).$ It is clear that $g(\cdot)$ is $S^{1}$-almost automorphic and $\hat{q}\in C_{0}([0,\infty) : L^{1}([0,1]: X))$ because $L^{p(x)}([0,1]: X) \hookrightarrow L^{1}([0,1]: X);$ therefore, $f\in AAAS^{1}( [0,\infty) : X).$ Using the fact that $ L^{p(x)}_{S}([0,\infty) : X) \hookrightarrow L^{1}_{S}([0,\infty) : X),$ it readily follows that there is a finite constant $c>0,$ independent of $f(\cdot),$ such that
$$
\bigl\|f\bigr\|_{L_{S}^{p(x)}([0,\infty) : X)} \leq c \bigl\|f\bigr\|_{L_{S}^{1}([0,\infty) : X)}.
$$ 
This completes the proof.
\end{proof}

{\sc Problem.} In \cite{toka-marek}, we have proved the following:
If $p\in D_{+}([0,1]),$ then 
\begin{align*}
L^{\infty}({\mathbb R}:X) \cap APS^{p(x)}({\mathbb R} : X)=L^{\infty}({\mathbb R}:X) \cap APS^{1}({\mathbb R}: X)
\end{align*}
and
\begin{align*}
L^{\infty}( [0,\infty):X) \cap AAPS^{p(x)}( [0,\infty) : X)=L^{\infty}( [0,\infty):X) \cap AAPS^{1}( [0,\infty): X).
\end{align*}
The proof given in the above-mentioned paper does not work for almost automorphy with variable exponent. Because of that, we would like to ask whether the assumption 
$p\in D_{+}([0,1])$ implies 
\begin{align*}
L^{\infty}({\mathbb R}:X) \cap AAS^{p(x)}({\mathbb R} : X)=L^{\infty}({\mathbb R}:X) \cap AAS^{1}({\mathbb R}: X)
\end{align*}
and
\begin{align*}
L^{\infty}( [0,\infty):X) \cap AAAS^{p(x)}( [0,\infty) : X)=L^{\infty}( [0,\infty):X) \cap AAAS^{1}( [0,\infty): X)?
\end{align*}

The subsequent lemma can be deduced following the lines of proof of \cite[Proposition 3.5]{toka-marek}:

\begin{lem}\label{lepo}
Assume that $p\in {\mathcal P}([0,1])$ and $q\in C_{0}([0,\infty) : X).$ Then $q\in L_{S}^{p(x)}([0,\infty) : X)$ and $\hat{q}\in C_{0}([0,\infty) : L^{p(x)}([0,1] : X)).$
\end{lem}

For the sequel, it is worth noting that, due to an elementary line of reasoning, we have $AA_{c}({\mathbb R} : X)=AAS^{\infty}({\mathbb R}:X) \cap C_{b}({\mathbb R} : X).$ Hence, the function $f(\cdot)$ cannot belong to the class $AAS^{p(x)}({\mathbb R} : X)$ if $f(\cdot)$ is almost automorphic, not compactly almost automorphic, and $p(x)\equiv \infty,$ $x\in [0,1]$. Similarly, if $g(\cdot)$ is almost automorphic, not compactly almost automorphic and $q\in C_{0}([0,\infty) : X),$ then the function $f\equiv g+q$ cannot belong to the class $AAAS^{\infty}([0,\infty) : X).$ In the following proposition, we will find a simple sufficient condition on $p\in {\mathcal P}([0,1])$ ensuring that 
an (asymptotically) almost automorphic function is (asymptotically) $S^{p(x)}$-almost automorphic:

\begin{prop}\label{propa}
Let $p\in {\mathcal P}([0,1]),$ let $f : {\mathbb R} \rightarrow X$ be almost automorphic, resp., $f : [0,\infty) \rightarrow X$ be asymptotically almost automorphic, and let 
\begin{align}\label{salenjak-auto}
\int^{1}_{0}\lambda^{p(x)}\, dx<\infty\mbox{ for all }\lambda>0.
\end{align}
Then $f(\cdot)$ is $S^{p(x)}$-almost automorphic, resp., $f(\cdot)$ is asymptotically $S^{p(x)}$-almost automorphic.
\end{prop}

\begin{proof}
The argumentation used in almost periodic case shows that $f(\cdot)$ is $S^{p(x)}$-bounded and $\|f\|_{L_{S}^{p(x)}}\leq \|f\|_{\infty}.$ 
Let $(b_{n})$ be a given real sequence. Then there exist a subsequence $(a_{n})$ of $(b_{n})$ and a map $g : {\mathbb R} \rightarrow X$ such that
\eqref{first-equ} holds, pointwise for $t\in {\mathbb R}.$ It is well known that $g\in L^{\infty}({\mathbb R} : X)$ and, by \cite[Proposition 3.6(i)]{toka-marek}, $g\in  L_{S}^{p(x)}({\mathbb R}:X).$ Due to \eqref{salenjak-auto}, we get that the Lebesgue measure of the set $\{ x\in [0,1] : p(x)=\infty\}$ is equal to zero and therefore any essentially bounded function $h: {\mathbb R} \rightarrow X$ satisfies that, for every $\lambda>0,$ we have $\int^{1}_{0}\varphi_{p(x)}(\|h(x)\|/\lambda)\, dx <\infty .$ Using this fact, we can apply
Lemma \ref{aux}(iii) in order to see that
$$
\lim_{n\rightarrow \infty}\bigl \| f\bigl( t+a_{n}+\cdot\bigr)-g(t+\cdot)\bigr\|_{L^{p(x)}([0,1] : X)}=0
$$
and  
$$
\lim_{n\rightarrow \infty}\bigl \|g\bigl(t+\cdot-a_{n}\bigr)-f(t+\cdot)\bigr\|_{L^{p(x)}([0,1] : X)}=0,
$$
pointwise for $t\in {\mathbb R}.$
This completes the proof of proposition for $S^{p(x)}$-almost automorphy; the corresponding result for asymptotical $S^{p(x)}$-almost automorphy follows by combining this and Lemma \ref{lepo}. 
\end{proof}

Now we will continue our analyses from \cite[Example 3.11]{toka-marek}:

\begin{example}\label{ogran-prop}
Set sign$(0):=0.$ Then, for every almost periodic function $f: {\mathbb R} \rightarrow {\mathbb R},$ we know that the function $F(\cdot):=$sign$(f(\cdot))$ is Stepanov $p(x)$-almost periodic
for any $p\in D_{+}([0,1])$ as well as that the function $F(\cdot)$ is Stepanov $p(x)$-bounded for any $p\in {\mathcal P}([0,1]);$ see \cite{toka-marek}. By Proposition \ref{oze}, we have that 
$F\in AAS^{p(x)}({\mathbb R} : {\mathbb C})$ for any $p\in D_{+}([0,1]).$

In \cite{toka-marek}, we have further analyze the special case that $f(x):=\sin x+\sin \sqrt{2}x,$ $x\in {\mathbb R}$ and $p(x):=1-\ln x,$ $x\in [0,1],$ showing that $F\notin APS^{p(x)}({\mathbb R} : {\mathbb C}).$ Now we will verify that $F\notin AAS^{p(x)}({\mathbb R} : {\mathbb C}).$ For this, it is sufficient to construct a real sequence $(a_{n})$ such that, for every $n\in {\mathbb N}$ and every $\lambda \in (0,2/e),$ we have
$$
\int^{1}_{0}\Biggl| \frac{F(x+a_{2n})-F(x+a_{2n-1})}{\lambda}\Biggr|^{1-\ln x}\, dx=\infty ;
$$
(see \eqref{ne-mu-je-stepanov-p(x)} with $t=0,$ and observe that in this case the sequence $F(a_{n_{k}}+\cdot)_{k\in {\mathbb N}}$ needs to be a Cauchy one in $L^{p(x)}[0,1]$). But, we have proved that, for any $\lambda \in (0,2/e)$, any $l>0,$   any interval $I\subseteq {\mathbb R} \setminus \{0\}$ of length $l>0$ and any $\tau \in I,$ there exists $t\in {\mathbb R}$ such that 
\begin{align}\label{salenjak}
\int^{1}_{0} \Bigl( \frac{1}{\lambda}\Bigr)^{1-\ln x}&\Bigl| F  (x+t+\tau) -F(x+t) \Bigr|^{1-\ln x}\, dx =\infty.
\end{align}
Let $\lambda \in (0,2/e)$ be arbitrarily chosen, $l=1,$ $I=I_{n}=[n,n+1]$ and $\tau =n.$ Then we can find $t=t_{n}$ such that \eqref{salenjak} holds, so that the claimed follows by plugging $a_{2n-1}:=t_{n}$ and $a_{2n}:=t_{n}+n$ ($n\in {\mathbb N}$). 

In the case that $f(x):=\sin x,$ $x\in {\mathbb R},$ we have also proved that the function $F(\cdot)$ is $S^{p(x)}$-almost periodic for any $p\in {\mathcal P}([0,1]).$ By Proposition \ref{oze}, we have that $F(\cdot)$ is $S^{p(x)}$-almost automorphic for any $p\in {\mathcal P}([0,1]).$
\end{example}

To the best knowledge of authors, in the existing literature concerning Stepanov almost automorphic functions, the authors have examined only such functions that are Stepanov $p$-almost automorphic for any exponent $p\in [1,\infty),$ and therefore, Stepanov $p(x)$-almost automorphic for any function $p\in D_{+}([0,1])$ (cf. Theorem \ref{toka-maremare}(ii)). Therefore, it is natural to ask whether there exists a  Stepanov almost automorphic function that is not Stepanov $p$-almost automorphic for certain exponent $p\in (1,\infty).$ The answer is affirmative and, without going into full problematic concerning this and similar questions, we would like to recall that
H. Bohr and E. F$\o$lner have constructed, for any given number $p>1$, a Stepanov almost periodic function defined on the whole real axis that is Stepanov $p$-bounded and not Stepanov $p$-almost periodic (see \cite[Example, pp. 70-73]{bohr-folner}). This function, denoted here by $f(\cdot),$ is clearly Stepanov almost automorphic and now we will prove that $f(\cdot)$ cannot be Stepanov $p$-almost automorphic. Consider, for the sake of simplicity, the case that $h_{1}=2$ in the afore-mentioned theorem and suppose the contrary. Then it is well known that the mapping $\hat{f} : {\mathbb R} \rightarrow L^{p}([0,1]: X)$ is compactly almost automorphic. Since the class of almost automorphic functions coincides with the class of Levitan $N$-almost periodic functions (see e.g. \cite[p. 111]{diagana} and \cite[pp. 53-54]{levitan}), for every $\epsilon>0$ and $N>0,$ there exists a finite number $L>0$ such that any interval $I\subseteq {\mathbb R}$ contains a number $\tau \in I$ such that $\|\hat{f}(t\pm \tau) -\hat{f}(t)\|_{L^{p}([0,1]: X)}<\epsilon.$ Especially, with $\epsilon >0$ arbitrarily small and $N=3/2,$ we get the existence of a finite number $L>0$ such that 
any interval $I\subseteq {\mathbb R} \setminus [-1,1]$ contains a number $\tau \in I$ such that
$$
\int^{x+1}_{x}\bigl|f(s+\tau)-f(s)\bigr|^{p}\, ds <\epsilon^{p},\quad x\in [-3/2,3/2].
$$
With $x=-3/2,$ this implies 
$$
\int^{3/2}_{-3/2}\bigl|f(s+\tau)-f(s)\bigr|^{p}\, ds <2\epsilon^{p},
$$
which is in contradiction with the estimate $
\int^{3/2}_{-3/2}\bigl|f(s+\tau)-f(s)\bigr|^{p}\, ds\geq 2^{-p}$ (see \cite[p. 73, l.-9 - l.-4]{bohr-folner}).

\section{Generalized two-parameter almost automorphic type functions and composition principles}\label{dodato}

Suppose that $(Y,\|\cdot \|_{Y})$ is a complex Banach space, as well as that $I={\mathbb R}$ or $I=[0,\infty).$ 
By $C_{0}([0,\infty) \times Y : X)$ we denote the space consisting of all
continuous functions $h : [0,\infty)  \times Y \rightarrow X$ such that $\lim_{t\rightarrow \infty}h(t, y) = 0$ uniformly for $y$ in any compact subset of $Y .$ A continuous function $f : I  \times Y  \rightarrow X$ is called uniformly continuous on
bounded sets, uniformly for $t \in I$ iff for every $	\epsilon > 0$ and every bounded subset $K$ of $Y$ there
exists a number $\delta_{\epsilon,K }> 0$ such that $\|f(t , x)- f (t, y)\| \leq \epsilon$ for all $ t \in I$ and all $x,\ y \in K$ satisfying that $\|x-y\|_{Y}\leq \delta_{\epsilon,K }.$ If $f : I  \times Y  \rightarrow X,$ set $\hat{f}(t , y):=f(t +\cdot, y),$ $t\geq 0,$ $y\in Y.$

We need to recall the following well-known definition (see e.g. \cite{diagana} and \cite{nova-mono} for more details):

\begin{defn}\label{definicija}
Let $1\leq p <\infty.$
\begin{itemize}
\item[(i)]
A jointly continuous function $f : {\mathbb R} \times Y \rightarrow X$ is said to be almost automorphic 
iff
for every sequence of real numbers $(s_{n}')$ there exists a subsequence $(s_{n})$ such
that
$$
G(t, y) := \lim_{n\rightarrow \infty}F\bigl(t +s_{n}, y\bigr)
$$
is well defined for each $t \in {\mathbb R}$ and $y\in Y,$ and
$$
\lim_{n\rightarrow \infty}G\bigl(t -s_{n}, y\bigr) = F(t, y)
$$
for each $t \in {\mathbb R}$ and $y\in Y.$ The vector space consisting of such functions will be denoted by $AA({\mathbb R} \times Y :X).$
\item[(ii)] A bounded continuous function $f :  {\mathbb R} \times Y \rightarrow X$ is said to be
pseudo-almost automorphic iff $F = G +\Phi,$ where $G \in AA({\mathbb R} \times Y :X)$ and $\Phi \in
PAP_{0}({\mathbb R} \times Y:X);$ here, $PAP_{0}({\mathbb R} \times Y : X)$ denotes the space consisting of all continuous functions $\Phi : {\mathbb R} \times Y \rightarrow X$ such
that $\{\Phi (t,y) : t \in {\mathbb R} \} $ is bounded for all $y\in Y,$ and
$$
\lim_{r\rightarrow \infty} \frac{1}{2r}\int^{r}_{-r}\| \Phi(s,y)\|\, ds=0,
$$
uniformly in $y\in Y.$ The vector space consisting of such functions will be denoted by $PAA({\mathbb R} \times Y:X).$
\end{itemize}
\end{defn}

We introduce the notions of a Stepanov two-parameter $p(x)$-almost automorphic function and an asymptotically Stepanov two-parameter $p(x)$-almost automorphic function as follows:

\begin{defn}\label{stepanov-auto-wr}
Let $p\in {\mathcal P}([0,1]),$ and let $f : {\mathbb R} \times Y \rightarrow X$ be such that for each $y\in Y$ we have $f(\cdot ,y)\in L^{p(x)}_{S}({\mathbb R} : X).$ Then we say that $f(\cdot,\cdot)$ is Stepanov $p(x)$-almost automorphic
iff for every $y\in Y$ the mapping $f(\cdot,y)$ is $S^{p(x)}$-almost automorphic; that is, for any real sequence $(a_{n})$ there exist a subsequence $(a_{n_{k}})$ of $(a_{n})$ and a map $g : {\mathbb R} \times Y \rightarrow X$ such that $g(\cdot,y)\in L_{S}^{p(x)}({\mathbb R}:X)$ for all $y\in Y$ as well as that:
\begin{align*}
\lim_{k\rightarrow \infty}\Bigl \| f\bigl(t+a_{n_{k}}+\cdot,y\bigr) -g(t+\cdot,y)\Bigr \|_{L^{p(x)}[0,1]}  =0
\end{align*}
and
\begin{align*}
\lim_{k\rightarrow \infty}\Bigl \| g\bigl( t+\cdot-a_{n_{k}},y\bigr) -f(t+\cdot,y)\Bigr \|_{L^{p(x)}[0,1]} =0
\end{align*}
for each $ t\in {\mathbb R}$ and for each $y\in Y.$ We denote by $AAS^{p(x)}({\mathbb R} \times Y : X)$ the vector space consisting of all such functions.
\end{defn}

\begin{defn}\label{lot}
Let $p\in {\mathcal P}([0,1]).$ A function $f : [0,\infty)  \times Y \rightarrow X$
is said to be asymptotically $S^{p(x)}$-almost automorphic
iff $\hat{f}: [0,\infty)  \times Y \rightarrow  L^{p(x)}([0,1]:X)$ is asymptotically almost automorphic. The collection of such functions will be denoted by $AAAS^{p(x)}([0,\infty) \times Y : X).$
\end{defn}

The following well-known result of Fan et al. \cite{fan-et-ali} is reformulated here for Stepanov $p(x)$-almost automorphy:

\begin{thm}\label{zen-et-al}
Assume that $p\in {\mathcal P}([0,1]),$ and $f\in AAS^{p(x)}({\mathbb R} \times Y : X).$ If there exists a constant $L>0$ such that for all $x,\ y\in L^{p(x)}_{S}({\mathbb R} : Y)$
\begin{align*}
\Bigl\| f(t+\cdot,x(\cdot))-f(t+\cdot,y(\cdot))\Bigr \|_{L^{p(x)}([0,1]:X)} \leq L_{Y}\bigl\| x(\cdot)-y(\cdot)\bigr \|_{L^{p(x)}([0,1] : Y)}
\end{align*}
then for each $x\in AAS^{p}({\mathbb R} : Y)$ with relatively compact range in $Y$ one has that $f(\cdot,x(\cdot))\in AAS^{p}({\mathbb R} : X).$
\end{thm}

The following result generalizes that one established by Ding et al. \cite{ding-xll} (see e.g. \cite[pp. 134-138]{diagana}) . The proof is similar and therefore omitted: 

\begin{thm}\label{vcb-show}
Let $I={\mathbb R}$, and let $p\in {\mathcal P}([0,1]).$
Suppose that the following conditions hold:
\begin{itemize}
\item[(i)] $f \in AAS^{p(x)}(I \times Y : X)  $ and there exist a function $r\in {\mathcal P}([0,1])$ such that $ r(\cdot)\geq \max (p(\cdot), p(\cdot)/p(\cdot) -1)$ and a function $ L_{f}\in L_{S}^{r(x)}(I) $ such that:
\begin{align}\label{vbnmp}
\| f(t,x)-f(t,y)\| \leq L_{f}(t)\|x-y\|_{Y},\quad t\in I,\ x,\ y\in Y;
\end{align}
\item[(ii)] $u \in AAS^{p(x)} (I: Y),$ and there exists a set ${\mathrm E} \subseteq I$ with $m ({\mathrm E})= 0$ such that
$ K :=\{u(t) : t \in I \setminus {\mathrm E}\}$
is relatively compact in $Y;$ here, $m(\cdot)$ denotes the Lebesgue measure.
\end{itemize}
Define $q\in {\mathcal P}([0,1])$ by
$q(x):=p(x)r(x)/p(x)+r(x),$ if $x\in [0,1]$ and $r(x)<\infty,$ $q(x):=p(x),$ if $x\in [0,1]$ and $r(x)=\infty.$ Then $q(x)\in [1, p(x))$ for $x\in [0,1],$ $r(x)<\infty$ and $f(\cdot, u(\cdot)) \in AAS^{q(x)}(I : X).$
\end{thm}

Concerning asymptotical two-parameter Stepanov $p(x)$-almost automorphy, we can deduce the following composition principle with $X=Y;$ see the proofs of \cite[Proposition 2.7.3, Proposition 2.7.4]{nova-mono} for more details:

\begin{prop}\label{bibl}
Let $I =[0,\infty),$ and let $p\in {\mathcal P}([0,1]).$
Suppose that the following conditions hold:
\begin{itemize}
\item[(i)] $g \in AAS^{p(x)}(I \times X : X)  ,$ there exist a function $r\in {\mathcal P}([0,1])$ such that $ r(\cdot)\geq \max (p(\cdot), p(\cdot)/p(\cdot) -1)$ and a function $ L_{g}\in L_{S}^{r(x)}(I) $ such that \eqref{vbnmp} holds with the function $f(\cdot, \cdot)  $ replaced by the function $g(\cdot, \cdot)  $ therein.
\item[(ii)] $v \in AAS^{p(x)}(I:X),$ and there exists a set ${\mathrm E} \subseteq I$ with $m ({\mathrm E})= 0$ such that
$ K =\{v(t) : t \in I \setminus {\mathrm E}\}$
is relatively compact in X.
\item[(iii)] $f(t,x)=g(t,x)+q(t,x)$ for all $t\geq 0$ and $x\in X,$ where $\hat{q}\in C_{0}([0,\infty) \times X : L^{q(x)}([0,1]:X))$
with $q(\cdot)$ defined as above;
\item[(iv)] $u(t)=v(t)+\omega(t) $ for all $t\geq 0,$ where $\hat{\omega}\in C_{0}([0,\infty) : L^{p(x)}([0,1]:X)).$
\item[(v)]  There exists a set $E' \subseteq I$ with $m (E')= 0$ such that
$ K' =\{u(t) : t \in I \setminus E'\}$
is relatively compact in $ X.$
\end{itemize}
Then $f(\cdot, u(\cdot)) \in AAAS^{q(x)}(I : X).$
\end{prop}

\section{Generalized (asymptotical) almost automorphy in Lebesgue spaces with variable exponents
$L^{p(x)}:$ actions of convolution products and some applications}\label{raske}

We start this section by stating
the following generalization of \cite[Proposition 5]{element}  (the reflexion at zero keeps the spaces of Stepanov $p$-almost automorphic functions unchanged, which may or may not be the case with the spaces of Stepanov $p(x)$-almost automorphic functions):

\begin{prop}\label{ravi-and-variable}
Suppose that $p\in D_{+}([0,1]),$ $q\in {\mathcal P}([0,1]),$ $1/p(x) +1/q(x)=1$
and $(R(t))_{t> 0}\subseteq L(X,Y)$ is a strongly continuous operator family satisfying that
$M:=\sum_{k=0}^{\infty}\|R(\cdot +k)\|_{L^{q(x)}[0,1]}<\infty .$ If $\check{g} : {\mathbb R} \rightarrow X$ is $S^{p(x)}$-almost automorphic, then the function $G: {\mathbb R} \rightarrow Y,$ given by
\begin{align}\label{wer}
G(t):=\int^{t}_{-\infty}R(t-s)g(s)\, ds,\quad t\in {\mathbb R},
\end{align}
is well-defined and almost automorphic.
\end{prop}

\begin{proof}
The proof of theorem is very similar to that of above-mentioned proposition since the H\"older inequality holds in our framework (see Lemma \ref{aux}(ii)) and any element of $L^{p(x)}([0,1]:X)$ is absolutely continuous with respect to the norm $\|\cdot\|_{L^{p(x)}}$ (see \cite[Definition 1.12, Theorem 1.13]{fan-zhao}), which clearly implies that the translation mapping $t\mapsto \check{g}(\cdot -t)\in L^{p(x)}([0,1] : X),$ $t\in {\mathbb R}$ is continuous (we need this fact for proving the continuity of mapping $F_{k}(\cdot)$ appearing in the proof of \cite[Proposition 5]{element}, $k\in {\mathbb N}$). The remaining part of proof can be given by copying the corresponding part of proof of above-mentioned proposition. 
\end{proof}

In general case $p\in {\mathcal P}([0,1]),$ there are elements of $L^{p(x)}([0,1])$ that are not absolutely continuous with respect to the norm $\|\cdot\|_{L^{p(x)}};$ see e.g. \cite[p. 602]{rakosnik}. In this case, Proposition \ref{ravi-and-variable} continues to hold if we impose the condition on continuity of mapping $t\mapsto \check{g}(\cdot -t)\in L^{p(x)}([0,1] : X),$ $t\in {\mathbb R}$ in place of condition $p\in D_{+}([0,1]).$

Proposition \ref{ravi-and-variable} can be simply incorporated in the study of existence and uniqueness of almost periodic solutions of the following abstract Cauchy differential inclusion of first order
\begin{align}\label{decko-leftt}
u^{\prime}(t)\in {\mathcal A}u(t)+g(t),\quad t\in {\mathbb R}
\end{align}
and the following abstract Cauchy relaxation differential inclusion
\begin{align}\label{decko-left}
D_{t,+}^{\gamma}u(t)\in -{\mathcal A}u(t)+g(t),\ t\in {\mathbb R},
\end{align}
where ${\mathcal A}$ is an MLO satisfying the condition (P), $D_{t,+}^{\gamma}$ denotes the Weyl-Liouville fractional derivative of order $\gamma \in (0,1)$ and $g: {\mathbb R} \rightarrow X$ satisfies certain assumptions; see \cite{toka-marek} and \cite{nova-mono} for further information in this direction.

In the following proposition, we state some invariance properties of generalized asymptotical almost automorphy in Lebesgue spaces with variable exponents
$L^{p(x)}$ under the action of finite convolution products. This proposition generalizes \cite[Proposition 6]{element} provided that $p>1$ in its formulation.

\begin{prop}\label{stepanov-almost-automorphy-p(x)}
Suppose that $ p,\ q\in C_{+}([0,1]),$ $1/p(x) +1/q(x)=1$
and $(R(t))_{t> 0}\subseteq L(X)$ is a strongly continuous operator family satisfying that, for every
$t\geq 0,$ we have that
$m_{t}:=\sum_{k=0}^{\infty}\|R(\cdot +t+k)\|_{L^{q(x)}[0,1]}<\infty .$
Suppose, further, that $\check{g} : {\mathbb R} \rightarrow X$ is $S^{p(x)}$-almost automorphic, 
$q\in L_{S}^{p(x)}( [0,\infty) :X)$ and $f(t)=g(t)+q(t),$ $t\geq 0.$ Let $r_{1},\ r_{2}\in {\mathcal P}([0,1])$ and
the following hold:
\begin{itemize}
\item[(i)]  For every $t\geq 0$, the mapping $x\mapsto \int^{t+x}_{0}R(t+x-s) q(s)\, ds,$ $x\in [0,1]$ belongs to the space $ L^{r_{1}(x)}([0,1] : X) $ and we have
\begin{align}\label{lep-primer}
\lim_{t\rightarrow +\infty}\Biggl\| \int^{t+x}_{0}R(t+x-s) q(s)\, ds\Biggr\|_{L^{r_{1}(x)}[0,1]}=0.
\end{align}
\item[(ii)] For every $t\geq 0,$ the mapping  $x\mapsto m_{t+x},$ $x\in [0,1]$ belongs to the space $L^{r_{2}(x)}[0,1]$ and we have
$$
\lim_{t\rightarrow +\infty}\bigl| m_{t+x}\bigr|_{L^{r_{2}(x)}[0,1]}=0.
$$
\end{itemize}
Then the function $H(\cdot),$ given by
\begin{align*}
H(t):=\int^{t}_{0}R(t-s)f(s)\, ds,\quad t\geq 0,
\end{align*}
is well-defined, bounded and belongs to the class $AAS^{p(x)}({\mathbb R} : X)+S^{r_{1}(x)}_{0}([0,\infty):X) + S^{r_{2}(x)}_{0}([0,\infty):X),$ with the meaning clear.
\end{prop}

\begin{proof}
Define $G(\cdot)$ by \eqref{wer} and $F(\cdot)$ by
\begin{align*}
F(t):=\int^{t}_{0}R(t-s)q(s)\, ds-\int^{\infty}_{t}R(s)g(t-s)\, ds:=F_{1}(t)+F_{2}(t),\quad t\geq 0.
\end{align*}
It can be simply shown that the function 
$F_{1}(\cdot)$ is well-defined and bounded because $q(\cdot)$ is $S^{p(x)}$-bounded and 
$m_{0}<\infty ;$ 
cf. the proof of \cite[Proposition 3.14]{toka-marek}. Furthermore, the integral $\int^{\infty}_{t}R(s)g(t-s)\, ds=F_{2}(t)$ is well-defined for all $ t\geq 0,$ which follows by applying Lemma \ref{aux}(ii):
\begin{align*}
\Biggl\| \int^{\infty}_{t}R(s)& g(t-s)\, ds\Biggr\|=\Biggl\| \int^{\infty}_{t}R(s) \check{g}(s-t)\, ds\Biggr\|
\\ & \leq \sum_{k=0}^{\infty}\int^{t+k+1}_{t+k}\| R(s) \| \| \check{g}(s-t)\|\, ds
\\&  =\sum_{k=0}^{\infty}\int^{1}_{0}\|R(s+t+k)\| \| \check{g}(s+k)\|\, ds
\\ & \leq 2 \bigl\|\check{g}\bigr\|_{L^{p(x)}_{S}({\mathbb R} : X)}\sum_{k=0}^{\infty}\|R(t+k+\cdot)\|_{L^{q(x)}[0,1]}
\\& \leq 2 \bigl\|\check{g}\bigr\|_{L^{p(x)}_{S}({\mathbb R} : X)}m_{t}<\infty,\quad t\geq 0.
\end{align*} 
Since $H(t)=G(t)+F(t)$ for all $ t\geq 0,$
we get that the function $H(\cdot)$ is well-defined and bounded;
due to Proposition \ref{ravi-and-variable}, it remains to be shown that the mapping $\hat{F_{i}}: [0,\infty) \rightarrow L^{r_{i}(x)} ([0,1] :X)$ is in class $C_{0}([0,\infty) : L^{r_{i}(x)}([0,1]:X))$ for $i=1,2.$ 
Let $k\in {\mathbb N}_{0}.$ For the continuity of mapping $t\mapsto F_{k,2}(t):=\int^{t+k+1}_{t+k}R(s)g(t-s)\, ds$, $t\geq 0,$ 
let us assume that $(t_{n})$ is a sequence of positive reals converging to some fixed number $t\geq 0.$ 
Having in mind Lemma \ref{aux}(ii), we obtain that
\begin{align*}
\bigl\| F_{k,2}(t_{n})-F_{k,2}(t)\bigr\|\leq 2\bigl\|\check{g}\bigr\|_{L^{p(x)}_{S}({\mathbb R} : X)}\bigl\| R(t_{n}+k+\cdot)-R(t+k+\cdot)\bigr\|_{L^{q(x)}[0,1]},\ n\in {\mathbb N},
\end{align*}
so that the claimed assertion follows by applying \cite[Theorem 1.13]{fan-zhao} (observe that we need the condition $ p,\ q\in C_{+}([0,1])$ here).
Using the condition $m_{t}<\infty$ as well as  
the Weierstrass criterion (see also the proof of \cite[Proposition 5]{element}), we get that the mapping $F_{2}(\cdot)$ is continuous. The continuity of mapping $t\mapsto \int^{t}_{0}R(t-s)q(s)\, ds,$ $t\geq 0$ can be shown similarly. By \cite[Proposition 3.7(iii)]{toka-marek}, we get that the mapping  $\hat{F_{i}}: [0,\infty) \rightarrow L^{r_{i}(x)} ([0,1] :X)$ is continuous for $i=1,2.$
Taking into account (i)-(ii) and the computation used above  for proving the boundedness of function $F_{2}(\cdot)$, we easily get that $\lim_{t\rightarrow +\infty}\|F_{i}(t+\cdot)\|_{ L^{r_{i}(x)} ([0,1] :X)}=0$ for $i=1,2.$
The proof of the proposition is thereby complete.
\end{proof}

The next example exhibits the use of ergodic Stepanov components with variable exponents (see also Example \ref{illust} below): 

\begin{example}\label{gotyie}
Suppose that $p(x)\equiv r_{2}(x)\equiv p>1,$ $q(x)\equiv p/p-1 >1$ ($x\in [0,1]$), $r_{1}\in {\mathcal P}([0,1])$,
$(R(t))_{t\geq 0}$ is strongly continuous, exponentially decaying, $g : {\mathbb R} \rightarrow X$ is $S^{p}$-almost periodic, $q : [0,\infty) \rightarrow X$ is $S^{p}$-bounded, \eqref{lep-primer} holds
but 
\begin{align*}
\lim_{t\rightarrow +\infty}\Biggl\| \int^{t+x}_{0}R(t+x-s) q(s)\, ds\Biggr\|_{L^{p}[0,1]}\neq 0.
\end{align*}
Then the function $H(\cdot)$ is bounded and belongs to the class $AAAS^{p(x)}([0,\infty) : X)+S^{r_{1}(x)}_{0}([0,\infty):X)$ but not to the class $AAAS^{p(x)}([0,\infty) : X);$ see also \cite[Remark 2.14(i)]{EJDE}.
\end{example}

We can simply apply Proposition \ref{stepanov-almost-automorphy-p(x)} in the analysis of existence and uniqueness of asymptotically $S^{r(x)}$-almost automorphic solutions for a wide class of abstract Volterra integro-differential equations and inclusions. For example,
Proposition \ref{stepanov-almost-automorphy-p(x)} is applicable in the analysis of asymptotically $S^{r(x)}$-almost automorphic solutions of the following
abstract integro-differential inclusion:
\begin{align*}
\Bigl[u(\cdot)&- \bigl(
g_{\zeta+1+i}\ast f\bigr)(\cdot)Cx\Bigr]  
\\ & +\sum
\limits_{j=1}^{n-1}c_{j}g_{\alpha_{n}-\alpha_{j}}\ast \Bigl[
u(\cdot)- \bigl(g_{\zeta+1+i}\ast f\bigr)(\cdot)Cx \Bigr]
\\ &+\sum \limits_{j\in {{\mathbb N}_{n-1}}\setminus D_{i}}c_{j}\bigl[g_{\alpha_{n}-\alpha_{j}+i+\zeta+1}\ast
f\bigr](\cdot)Cx \in
{\mathcal A}\bigl[g_{\alpha_{n}-\alpha}\ast u\bigr](\cdot),
\end{align*}
where $\zeta \geq 0$ is appropriately chosen, $C\in L(X)$ commutes with ${\mathcal A},$ $c_{j}\geq 0$ for $1\leq j\leq n-1,$ $0 \leq \alpha_{1}<\cdot \cdot
\cdot<\alpha_{n},$ $0\leq \alpha<\alpha_{n}$ and $f(\cdot)$ 
satisfies the requirements of Proposition \ref{stepanov-almost-automorphy-p(x)} (cf. \cite{gaston-marko} for the notion of sets $D_{i}$ and more details on the subject). 

In what follows, we will briefly explain how one can apply Proposition \ref{stepanov-almost-automorphy-p(x)} in the study of qualitative analysis of solutions of the following fractional relaxation inclusion
\[
\hbox{(DFP)}_{f,\gamma} : \left\{
\begin{array}{l}
{\mathbf D}_{t}^{\gamma}u(t)\in {\mathcal A}u(t)+f(t),\ t> 0,\\
\quad u(0)=x_{0},
\end{array}
\right.
\]
where ${\mathbf D}_{t}^{\gamma}$ denotes the Caputo fractional derivative of order $\gamma \in (0,1],$ $x_{0}\in X$ and
$f : [0,\infty) \rightarrow X$ satisfies certain properties.
Let $(S_{\gamma}(t))_{t>0}$ and $(P_{\gamma}(t))_{t>0}$ be the operator families defined in \cite{toka-marek}. Then we have
the existence of two finite constants $M_{1}>0$ and $M_{2}>0$ such that
\begin{align}\label{debil}
\bigl\| S_{\gamma}(t) \bigr\|+\bigl\| P_{\gamma}(t) \bigr\|\leq M_{1}t^{\gamma (\beta-1)},\quad t>0
\end{align}
and
\begin{align}\label{debil-prim}
\bigl\| S_{\gamma}(t) \bigr\|\leq M_{2}t^{-\gamma},\ t\geq 1, \ \ \mbox{   }\ \ \bigl\| P_{\gamma}(t) \bigr\|\leq M_{2}t^{-2\gamma },\ t\geq 1.
\end{align}
Set $R_{\gamma}(t):= t^{\gamma -1}P_{\gamma}(t),$ $t>0.$
By a mild solution of (DFP)$_{f,\gamma},$ we mean any function $u\in C([0,\infty) : X)$ satisfying that
$$
u(t)=S_{\gamma}(t)x_{0}+\int^{t}_{0}R_{\gamma}(t-s)f(s)\, ds,\quad t\geq 0.
$$
The estimates \eqref{debil}-\eqref{debil-prim} and the representation formula for $u(\cdot)$ are crucial for applications of Proposition \ref{stepanov-almost-automorphy-p(x)}. We provide below an 
illustrative example:

\begin{example}\label{illust}
Let $x_{0}\in X$ belong to the domain of continuity of $(T(t))_{t>0},$ i.e., $\lim_{t\rightarrow 0+}T(t)x=x.$ Then we know \cite{nova-mono} that $\lim_{t\rightarrow 0+}S_{\gamma}(t)x=x$ so that the mapping $t\mapsto S_{\gamma}(t)x,$ $t\geq 0$ is continuous and tends to zero as $t\rightarrow +\infty.$ Let $p,\ q\in (1,\infty),$ and let $q(\gamma \beta -1)>-1.$ Assume that $p(x)\equiv r_{2}(x)\equiv p$ ($x\in [0,1]$) and $r_{1}\in {\mathcal P}([0,1]).$ Then $\|R(\cdot)\|_{L^{q}[0,1]}<\infty$ and the computation similar to that one established in \cite[Remark 2.14(ii)]{EJDE} shows that $m_{t}<\infty$ for all $t\geq 0$ as well as that the mapping $t\mapsto m_{t},$ $t\geq 0$ is continuous and satisfies $m_{t}\leq \mbox{Const.}t^{\nu (-1-\gamma)},$ $t\geq 1,$ where $\nu \in (0,1)$ is chosen so that $(1-\nu)(1+\gamma)>1.$ By Lemma \ref{lepo}, we get $\hat{m_{\cdot}}\in C_{0}([0,\infty) : L^{r_{2}(x)}([0,1] : X)).$ Writing the first integral in (i) of Proposition \ref{stepanov-almost-automorphy-p(x)} as $\int^{t+x}_{0}R(t+x-s) q(s)\, ds=\int^{1}_{0}R(s) q(t+x-s)\, ds+\int^{t+x}_{1}R(s) q(t+x-s)\, ds$ ($t\geq 0,$ $x\in [0,1]$), and using the growth order of $\|R(\cdot)\|$, it can be simply shown that the validity of condition
$$
\lim_{t\rightarrow +\infty}\Biggl\| \bigl\|q(t+x-\cdot)\bigr\|_{L^{p}[0,1]}+ \int^{t+x-1}_{0}\frac{\|q(s)\|}{1+s^{\gamma}}\, ds\Biggr\|_{L^{r_{1}(x)}[0,1]}=0
$$
yields that, for every $t\geq 0,$ the mapping $x\mapsto \int^{t+x}_{0}R(t+x-s) q(s)\, ds,$ $x\in [0,1]$ belongs to the space $ L^{r_{1}(x)}([0,1] : X) $ and 
\eqref{lep-primer} holds. Therefore, Proposition \ref{stepanov-almost-automorphy-p(x)} is applicable.
\end{example}

\end{document}